\documentclass{amsart}
\usepackage{amssymb,latexsym}
\usepackage{amscd,amsthm}

\newtheorem{theorem}{Theorem}[section]
\newtheorem{lemma}[theorem]{Lemma}

\newtheorem{proposition}[theorem]{Proposition}
\newtheorem{corollary}[theorem]{Corollary}

\theoremstyle{definition}
\newtheorem{definition}[theorem]{Definition}

\newtheorem*{remark}{Remark}

\DeclareMathOperator{\Ext}{Ext}
\DeclareMathOperator{\Hom}{Hom}
\DeclareMathOperator{\Tor}{Tor}

\DeclareMathOperator{\cok}{cok}

\newcommand{\cat}[1]{\mathcal{#1}}           

\newcommand{\tensor}{\otimes}

\newcommand{\class}[1]{\mathcal{#1}}   

\newcommand{\Z}{\mathbb{Z}}
\newcommand{\Q}{\mathbb{Q/Z}}
\newcommand{\mathcolon}{\colon\,} 

\newcommand{\ch}{\textnormal{Ch}(R)}
\newcommand{\cha}[1]{\textnormal{Ch}(\mathcal{#1})}
\newcommand{\rmod}{R\text{-Mod}}

\newcommand{\tilclass}[1]{\widetilde{\class{#1}}}
\newcommand{\dgclass}[1]{dg\widetilde{\class{#1}}}

\newcommand{\dwclass}[1]{dw\widetilde{\class{#1}}}
\newcommand{\exclass}[1]{ex\widetilde{\class{#1}}}

\newcommand{\rightperp}[1]{#1^{\perp}}
\newcommand{\leftperp}[1]{{}^\perp #1}

\newcommand{\homcomplex}{\mathit{Hom}}

\begin{document}

\title[Ding modules and complexes]{on Ding injective, Ding projective, and Ding flat modules and complexes}


\author{James Gillespie}
\address{Ramapo College of New Jersey \\
         School of Theoretical and Applied Science \\
         505 Ramapo Valley Road \\
         Mahwah, NJ 07430}
\email[Jim Gillespie]{jgillesp@ramapo.edu}
\urladdr{http://pages.ramapo.edu/~jgillesp/}

\date{\today}

\begin{abstract}
We characterize Ding modules and complexes over Ding-Chen rings. We show that over a Ding-Chen ring $R$, the Ding projective (resp. Ding injective, resp. Ding flat) $R$-modules coincide with the Gorenstein projective (resp. Gorenstein injective, resp. Gorenstein flat) modules, which in turn are nothing more than modules appearing as a cycle of an exact complex of projective (resp. injective, resp. flat) modules. We prove a similar characterization for chain complexes of $R$-modules: A complex $X$ is Ding projective (resp. Ding injective, resp. Ding flat) if and only if each component $X_n$ is Ding projective (resp. Ding injective, resp. Ding flat). Along the way, we generalize some results of Stovicek and Bravo-Gillespie-Hovey to obtain other interesting corollaries. For example, we show that over any Noetherian ring, any exact chain complex with Gorenstein injective components must have all cotorsion cycle modules. That is, $\Ext^1_R(F,Z_nI) = 0$ for any such complex $I$ and flat module $F$. On the other hand, over any coherent ring, the cycles of any exact complex $P$ with projective components must satisfy $\Ext^1_R(Z_nP,A) = 0$ for any absolutely pure module $A$.
\end{abstract}

\maketitle

\section{introduction}\label{sec-intro}
Let $R$ be a Gorenstein ring in the sense of~\cite{iwanaga, iwanaga2}.  This is a left and right Noetherian ring $R$ having finite injective dimension as both a left and right module over itself. Over such rings, exact chain complexes of projective, injective, and flat $R$-modules each have very nice homological properties. In particular, for an exact complex $P$ of projectives, the complex $\Hom_R(P,Q)$ is also exact for any projective module $Q$. In general, such complexes are called \emph{totally acyclic complexes of projectives} and the modules appearing as a cycle $Z_nP$ of such a complex are called \emph{Gorenstein projective}. So over Gorenstein rings, any exact complex of projectives is totally acyclic and their cycles are all Gorenstein projective. Similar statements hold for exact complexes of injectives and flats, and the corresponding \emph{Gorenstein injective} and \emph{Gorenstein flat} modules. Gorenstein homological algebra is the study of these modules and complexes and the theory is particularly satisfying over Gorenstein rings. In this case, most results from traditional homological algebra have an analog in Gorenstein homological algebra. The book~\cite{enochs-jenda-book} is a standard reference, but many authors have studied the subject.  

Stovicek recently proved the coherent analog of the above result concerning totally acyclic complexes of injectives~\cite[Prop.~7.9]{stovicek-purity}. It raises the same question for the projective and flat analogs, which is answered in this paper. Before explaining further, first recall that a (left) coherent ring is one in which all finitely generated (left) ideals are finitely presented. A ring that is both left and right coherent is called coherent and these include all Noetherian and Von Neumann regular rings. A lesson learned from~\cite{stenstrom-fp} is that many results in homological algebra extend from Noetherian to coherent rings by replacing finitely generated modules with finitely presented modules. In the process, injective modules are replaced with FP-injective (i.e.~absolutely pure) modules.
In this way, a Ding-Chen ring is a coherent ring $R$ which has finite absolutely pure (i.e.~FP-injective) dimension as both a left and right module over itself. They were introduced in~\cite{ding and chen 93, ding and chen 96}. Since a coherent ring is Noetherian if and only if the FP-injective modules coincide with the injective modules, the Ding-Chen rings are nothing more than Gorenstein rings whenever $R$ is Noetherian.

In~\cite{gillespie-ding}, the author noted that the work of Ding, Mao and Li,  see especially~\cite{ding and mao 05, ding and mao 07, ding and mao 08, ding and mao and li 09}, provides a natural way to extend notions from Gorenstein homological algebra from Noetherian to coherent rings. In the process, the Gorenstein modules are replaced by Ding's modules. For example, we say a module is \emph{Ding projective} if it is a cycle module of some exact complex $P$ of projectives such that $\Hom_R(P,F)$ remains exact for all flat modules $F$. See Definitions~ \ref{def-Ding injectives}, \ref{def-Ding projectives}, and~\ref{def-Ding flats}. It admittedly seems strange to require that $\Hom_R(P,F)$ remain exact for all flat modules $F$ rather than just all projectives. After all, it feels like we are requiring too much. However, it was shown in~\cite{bravo-gillespie-hovey} that the Ding projectives are the cofibrant objects of an especially nice Quillen model structure on $R$-Mod, the category of (left) $R$-modules, whenever $R$ is coherent. It implies that every module $M$ can be approximated by a Ding projective and that the full subcategory of all Ding projective modules naturally form a Frobenius category. We think of the associated stable category as the (projective) stable module category of $R$. The analog of this for the usual Gorenstein projectives over Noetherian rings simply doesn't seem to be true, in general. 

The first result of this paper is an extension, to Ding-Chen rings, of the well known result concerning Gorenstein rings described in the first paragraph. Part (2) is the result of Stovicek which motivated the analogous question for Ding projectives and Ding flats.

\begin{theorem}\label{them-1}
Let $R$ be a Ding-Chen ring. That is, a left and right coherent ring $R$ such that ${}_RR$ and $R_R$ each have finite absolutely pure dimension. Let $X$ be an exact chain complex. Then
\begin{enumerate}
\item If each component $X_n$ is projective, then each cycle $Z_nX$ is Ding projective. Indeed $\Hom_R(X,F)$ remains exact for all flat modules $F$. 
\item If each component $X_n$ is injective, then each cycle $Z_nX$ is Ding injective. Indeed $\Hom_R(A,X)$ remains exact for all absolutely pure $A$. 
\item If each component $X_n$ is flat, then each cycle $Z_nX$ is Ding flat. Indeed $A \otimes_R X$ remains exact for all absolutely pure $A$. 
\end{enumerate}
Consequently, the Gorenstein modules coincide with the Ding modules whenever $R$ is a Ding-Chen ring. 
\end{theorem}

Before discussing the proof methods we note that Gorenstein \emph{chain complexes} have also been studied quite a bit. For example, it has been known for some time that over Gorenstein rings, a chain complex $X$ is Gorenstein injective (resp. Gorenstein projective) if and only if each component is Gorenstein injective (resp. Gorenstein projective). See~\cite{garcia-rozas, enochs-estrada-iacob, GangYang-Liu models on complexes}. In~\cite{Liang-Liu, Ding-Chen-complex-models}, Yang, Liu, and Liang characterize the Ding complexes over a Ding-Chen ring. The next result is a surprising refinement of their characterization.

\begin{theorem}\label{them-2}
Let $R$ be a Ding-Chen ring and $X$ a chain complex of $R$-modules.
\begin{enumerate}
\item $X$ is Ding projective in $\ch$ if and only if each component $X_n$ is a Ding projective $R$-module.
\item $X$ is Ding injective in $\ch$ if and only if each component $X_n$ is a Ding injective $R$-module. 
\item $X$ is Ding flat in $\ch$ if and only if each component $X_n$ is a Ding flat $R$-module.
\end{enumerate}
Consequently, the Gorenstein complexes coincide with the Ding complexes whenever $R$ is a Ding-Chen ring. 
\end{theorem}

We now briefly describe our methods used and highlight other results in the paper. First, the injective cases are proved in a completely different fashion than their projective analogs, (while the flat analogs are relatively easy). For the injective case, we generalize the approach of Stovicek to record a useful result concerning cotorsion pairs and direct limits. If $(\class{W},\class{F})$ is a cotorsion pair, then $\class{W}$ is closed under direct limits whenever $\class{W}$ satisfies the two out of three property on short exact sequences; see Proposition~\ref{prop-direct limits}. The injective cases of Theorems~\ref{them-1} and~\ref{them-2} follow as corollaries, but the approach leads to a string of other interesting corollaries appearing in  Sections~\ref{sec-direct limits} and~\ref{sec-Ding injective}.  Perhaps the most interesting is Corollary~\ref{cor-cotorsions}. It implies that over \emph{any} Noetherian ring, any exact chain complex with Gorenstein injective components must have all cotorsion cycle modules. That is, $\Ext^1_R(F,Z_nI) = 0$ for any such complex $I$ and flat module $F$. 

The projective case requires a completely different approach than the injective one. Our proof of the projective part of Theorem~\ref{them-1} relies on a result proved in~\cite[Theorem~A.6]{bravo-gillespie-hovey}. To obtain the projective statement in Theorem~\ref{them-2} we in fact need to generalize this result to the category of chain complexes. Section~\ref{sec-complexes of projective complexes} is devoted to this, with Theorem~\ref{thm-dual-exact} being the main result. Finally, the following interesting result is proved in Theorem~\ref{them-projcycles}.

\begin{theorem}
Let $R$ be a coherent ring and $P$ any exact complex of projectives. Then $\Ext^1_R(Z_nP,A) = 0$ for any absolutely pure module $A$. 
\end{theorem}

Regarding the structure of the paper, after some preliminaries in Section~\ref{sec-prelim}, we first address the injective cases in Sections~\ref{sec-direct limits} and~\ref{sec-Ding injective}. Sections~\ref{sec-complexes of projective complexes} and~\ref{sec-Ding projectives} are devoted to the projective case, and Section~\ref{sec-flat} to the flat case.

\section{preliminaries}\label{sec-prelim}

Throughout the paper $R$ denotes a general ring with identity. An $R$-module will mean a left $R$-module, unless stated otherwise. The category of $R$-modules will be denoted $\rmod$. It is an abelian category.

\subsection{Chain complexes on abelian categories}\label{subsec-complexes}
Let $\cat{A}$ be an abelian category. We denote the corresponding category of chain complexes by $\cha{A}$. In the case $\cat{A} = \rmod$, we denote it by $\ch$. Our convention is that the differentials of our chain complexes lower degree, so $\cdots
\xrightarrow{} X_{n+1} \xrightarrow{d_{n+1}} X_{n} \xrightarrow{d_n}
X_{n-1} \xrightarrow{} \cdots$ is a chain complex. We also have the chain homotopy category of $\cat{A}$, denoted $K(\cat{A})$. Its objects are also chain complexes but its morphisms are chain homotopy classes of chain maps.
Given a chain complex $X$, the
\emph{$n^{\text{th}}$ suspension of $X$}, denoted $\Sigma^n X$, is the complex given by
$(\Sigma^n X)_{k} = X_{k-n}$ and $(d_{\Sigma^n X})_{k} = (-1)^nd_{k-n}$.
For a given object $A \in \cat{A}$, we denote the \emph{$n$-disk on $A$} by $D^n(A)$. This is the complex consisting only of $A \xrightarrow{1_A} A$ concentrated in degrees $n$ and $n-1$, and 0 elsewhere. We denote the \emph{$n$-sphere on $A$} by $S^n(A)$, and this is the complex consisting only of $A$ in degree $n$ and 0 elsewhere.

Given two chain complexes $X, Y \in \cha{A}$ we define $\homcomplex(X,Y)$ to
be the complex of abelian groups $ \cdots \xrightarrow{} \prod_{k \in
\Z} \Hom(X_{k},Y_{k+n}) \xrightarrow{\delta_{n}} \prod_{k \in \Z}
\Hom(X_{k},Y_{k+n-1}) \xrightarrow{} \cdots$, where $(\delta_{n}f)_{k}
= d_{k+n}f_{k} - (-1)^n f_{k-1}d_{k}$.
We get a functor
$\homcomplex(X,-) \mathcolon \cha{A} \xrightarrow{} \textnormal{Ch}(\Z)$. Note that this functor takes exact sequences to left exact sequences,
and it is exact if each $X_{n}$ is projective. Similarly the contravariant functor $\homcomplex(-,Y)$ sends exact sequences to left exact sequences and is exact if each $Y_{n}$ is injective. It is an exercise to check that the homology satisfies $H_n[Hom(X,Y)] = K(\cat{A})(X,\Sigma^{-n} Y)$.

Being an abelian category, $\cha{A}$ comes with Yoneda Ext groups. In particular, $\Ext^1_{\cha{A}}(X,Y)$ will denote the group of (equivalences classes) of short exact sequences $0 \xrightarrow{} Y \xrightarrow{} Z \xrightarrow{} X \xrightarrow{} 0$ under the Baer sum operation. There is a subgroup $\Ext^1_{dw}(X,Y) \subseteq \Ext^1_{\cha{A}}(X,Y)$ consisting of the ``degreewise split'' short exact sequences. That is,
those for which each $0 \xrightarrow{} Y_n \xrightarrow{} Z_n \xrightarrow{} X_n \xrightarrow{} 0$ is split exact. The following lemma gives a well-known connection between $\Ext^1_{dw}$ and the above hom-complex $\homcomplex$.

\begin{lemma}\label{lemma-homcomplex-basic-lemma}
For chain complexes $X$ and $Y$, we have isomorphisms:
$$\Ext^1_{dw}(X,\Sigma^{(-n-1)}Y) \cong H_n \homcomplex(X,Y) =
K(\cat{A})(X,\Sigma^{-n} Y)$$ In particular, for chain complexes $X$ and $Y$, $\homcomplex(X,Y)$ is
exact iff for any $n \in \Z$, any chain map $f \mathcolon \Sigma^nX \xrightarrow{} Y$ is
homotopic to 0 (or iff any chain map $f \mathcolon X \xrightarrow{} \Sigma^nY$ is homotopic
to 0).
\end{lemma}

In the case of $\cat{A} = R\textnormal{-Mod}$, we recall the usual tensor product of chain complexes. Given that $X$ (resp. $Y$) is a complex of right (resp. left) $R$-modules, the tensor product $X
\otimes Y$ is defined by $(X \otimes Y)_n = \oplus_{i+j=n} (X_i
\otimes Y_j)$ in degree $n$. The boundary map $\delta_n$ is defined
on the generators by $\delta_n (x \otimes y) = dx \otimes y +
(-1)^{|x|} x \otimes dy$, where $|x|$ is the degree of the element
$x$.

\subsection{The modified Hom and Tensor complexes}\label{subsec-modified hom and tensor} Here again, $\ch$ denotes the category of chain complexes of $R$-modules. The above $\homcomplex$ is often referred to as the \emph{internal hom}, for in the case that $R$ is commutative, $\homcomplex(X,Y)$ is again an object of $\ch$. Note that the cycles in degree 0 of the internal hom coincide with the \emph{external hom} functor: $Z_0[\homcomplex(X,Y)] \cong \Hom_{\ch}(X,Y)$. This idea in fact is useful to define an alternate internal hom as follows. Given $X, Y \in \ch$, we define $\overline{\homcomplex}(X,Y)$ to be the complex $$\overline{\homcomplex}(X,Y)_n = Z_n\homcomplex(X,Y)$$ with differential $$\lambda_n : \overline{\homcomplex}(X,Y)_n \xrightarrow{} \overline{\homcomplex}(X,Y)_{n-1}$$ defined by $(\lambda f)_k = (-1)^nd_{k+n}f_k$. Notice that the degree $n$ component of $\overline{\homcomplex}(X,Y)$ is exactly $\Hom_{\ch}(X,\Sigma^{-n}Y)$. In this way we get an internal hom $\overline{\homcomplex}$ which is useful for categorical considerations in $\ch$. For example, $\overline{\homcomplex}(X,-)$ is a left exact functor, and is exact if and only if $X$ is projective in the category $\ch$. On the other hand, $\overline{\homcomplex}(-,Y)$ is exact if and only if $Y$ is injective in $\ch$. There are corresponding derived functors which we denote by $\overline{\mathit{Ext}}^i$. They satisfy that $\overline{\mathit{Ext}}^i(X,Y)$ is a complex whose degree $n$ is $\Ext^i_{\ch}(X,\Sigma^{-n}Y)$.

Similarly, the usual tensor product of chain complexes does not characterize categorical flatness. For this we need the modified tensor product and its left derived torsion functor from~\cite{enochs-garcia-rozas} and~\cite{garcia-rozas}. We will denote it by $\overline{\otimes}$, and it is defined in terms of the usual tensor product $\otimes$ as follows. Given a complex $X$ of right $R$-modules and a complex $Y$ of left $R$-modules, we define $X \overline{\otimes} Y$ to be the complex whose $n^{\text{th}}$ entry is $(X \otimes Y)_n / B_n(X \otimes Y)$ with boundary map  $(X \otimes Y)_n / B_n(X \otimes Y) \rightarrow (X \otimes Y)_{n-1} / B_{n-1}(X \otimes Y)$ given by
\[
\overline{x \otimes y} \mapsto \overline{dx \otimes y}.
\]
This defines a complex and we get a bifunctor $ - \overline{\otimes} - $ which is right exact in each variable. We denote the corresponding left derived functors by $\overline{\Tor}_i$. We refer the reader to~\cite{garcia-rozas} for more details.

\subsection{Finitely generated projective complexes} A standard characterization of projective objects in $\ch$ is the following: A complex $P$ is \emph{projective} if and only if it is an exact complex with each cycle $Z_nP$ a projective $R$-module. We also recall that, by definition, a chain complex $X$ is \emph{finitely generated} if whenever $X = \Sigma_{i \in I} S_i$, for some collection $\{S_i\}_{i \in I}$ of subcomplexes of $X$, then there exists a finite subset $J \subseteq I$ for which $X = \Sigma_{i \in J} S_i$. It is a standard fact that $X$ is finitely generated if and only if it is bounded (above and below) and each $X_n$ is finitely generated. We say that a chain complex $X$ is of \textbf{type $\boldsymbol{FP_{\infty}}$} if it has a projective resolution by finitely generated projective complexes. Certainly any such $X$ is finitely presented and hence finitely generated.

\subsection{Absolutely clean and level complexes; character duality}\label{subsec-character duality}
The so-called level and absolutely clean modules were introduced in~\cite{bravo-gillespie-hovey} as generalizations of flat modules over coherent rings and injective modules over Noetherian rings. The same notions in the category $\ch$ were also studied in~\cite{bravo-gillespie}. Here we recall some definitions and results from~\cite{bravo-gillespie}.

\begin{definition}
We call a chain complex $A$ \textbf{absolutely clean} if $\Ext^1_{\ch}(X,A)=0$ for all chain complexes $X$ of type $FP_{\infty}$. Equivalently, if $\overline{\mathit{Ext}}^1(X,A) = 0$ for all complexes $X$ of right $R$-modules of type $FP_{\infty}$. On the other hand, we call a chain complex $L$ \textbf{level} if $\overline{\Tor}_1(X,L) = 0$ for all chain complexes $X$ of right $R$-modules of type $FP_{\infty}$.
\end{definition}

We refer to~\cite[Prop.~2.6/4.6]{bravo-gillespie} for proof of the following.

\begin{proposition}\label{prop-level chain complexes}
A chain complex $A$ is absolutely clean if and only if $A$ is exact and each $Z_nA$ is an absolutely clean $R$-module.
A chain complex $L$ is level if and only if $L$ is exact and each $Z_nL$ is a level $R$-module.
\end{proposition}

Recall that the character module of $M$ is defined as $M^+ = \Hom_{\Z}
(M, \Q )$, and that $M^+$ is a right (resp. left) $R$-module whenever $M$ is a left (resp. right) $R$-module.  The construction extends to chain complexes: Given a chain complex $X$, we have $X^+ = \Hom_{\Z}(X,\Q)$. Since $\Q$ is an injective cogenerator for the category of abelian groups, the functor $\Hom_{\Z}(-,\Q)$ preserves and reflects exactness. So Proposition~\ref{prop-level chain complexes} immediately gives us the following corollary due to the perfect character module duality between absolutely clean and level modules~\cite[Theorem~2.10]{bravo-gillespie-hovey}.

\begin{corollary}\label{cor-duality}
A chain complex $L$ of left (resp. right) modules is level if and only if $L^+ = \Hom_{\Z}(L,\Q)$ is an absolutely clean complex of right (resp. left) modules. And, a chain complex $A$ of left (resp. right) modules is absolutely clean if and only if $A^+ = \Hom_{\Z}(A,\Q)$ is a level complex of right (resp. left) modules.
\end{corollary}

\subsection{Cotorsion pairs} Let $\cat{A}$ be an abelian category.  By definition, a pair of classes $(\class{X},\class{Y})$ in $\cat{A}$ is called a \emph{cotorsion pair} if $\class{Y} = \rightperp{\class{X}}$ and $\class{X} = \leftperp{\class{Y}}$. Here, given a class of objects $\class{C}$ in $\cat{A}$, the right orthogonal  $\rightperp{\class{C}}$ is defined to be the class of all objects $X$ such that $\Ext^1_{\cat{A}}(C,X) = 0$ for all $C \in \class{C}$. Similarly, we define the left orthogonal $\leftperp{\class{C}}$. We call the cotorsion pair \emph{hereditary} if $\Ext^i_{\cat{A}}(X,Y) = 0$ for all $X \in \class{X}$, $Y \in \class{Y}$, and $i \geq 1$. The cotorsion pair is \emph{complete} if it has enough injectives and enough projectives. This means that for each $A \in \cat{A}$ there exist short exact sequences $0 \xrightarrow{} A \xrightarrow{} Y \xrightarrow{} X \xrightarrow{} 0$ and $0 \xrightarrow{} Y' \xrightarrow{} X' \xrightarrow{} A \xrightarrow{} 0$ with $X,X' \in \class{X}$ and $Y,Y' \in \class{Y}$.
Standard references include~\cite{enochs-jenda-book} and~\cite{trlifaj-book} and connections to abelian model categories can be found in~\cite{hovey} and~\cite{gillespie-hovey triples}.

Recall that a \emph{Grothendieck category} is a cocomplete abelian category with a set of generators and such that direct limits are exact. By an \emph{injective cotorsion pair} $(\class{W},\class{F})$ in a Grothendieck category $\cat{G}$ we mean a complete cotorsion pair for which $\class{W}$ is thick and $\class{W} \cap \class{F}$ is the class of injective objects.  Since Grothendieck categories have enough injectives it turns out that such a cotorsion pair is equivalent to an \emph{injective model structure} on $\cat{G}$. By this we mean the model structure is abelian in the sense of~\cite{hovey} and all objects are cofibrant. The fibrant objects in this case are exactly $\class{F}$ and the trivial objects are exactly $\class{W}$. See~\cite{gillespie-recollement} for more on injective cotorsion pairs. We also have the dual notion of \emph{projective cotorsion pairs} $(\class{C},\class{W})$ which give us \emph{projective model structures} on abelian categories with enough projectives.

\section{Injective cotorsion pairs and direct limits}\label{sec-direct limits}

Let $(\class{W},\class{F})$ be a cotorsion pair in $R$-Mod, or $\ch$, or some other Grothendieck category. Assume that $\class{W}$ is thick; that is, assume $\class{W}$ satisfies the two out of three property on short exact sequences. In this section we show that $\class{W}$ must be closed under direct limits. This leads to several interesting corollaries. The proof of the proposition below is adapted from Stovicek's~\cite[Prop.~5.3]{stovicek-purity}.  After studying the proof, the author simply realized that it can be reinterpreted to yield the following convenient result.

\begin{proposition}\label{prop-direct limits}
Let $(\class{W},\class{F})$ be a cotorsion pair in $R$-Mod with $\class{W}$ thick. Then $\class{W}$ is closed under direct limits.
\end{proposition}

\begin{proof}
As a first step we show that $\class{W}$ is closed under direct unions. It is in fact enough to prove this for well-ordered \emph{continuous} direct unions. So assume that we have a $\lambda$-sequence of module monomorphisms
$$X_0 \hookrightarrow{} X_1 \hookrightarrow{} X_2 \hookrightarrow{} \cdots \hookrightarrow{} X_{i} \hookrightarrow{} X_{i+1} \hookrightarrow{} \cdots$$
with $X_i  \in \class{W}$ for each $i < \lambda$. Then each $X_{i+1}/X_i \in \class{W}$ by the thickness hypothesis. Since the diagram of modules is assumed continuous we see that its colimit is the same thing as the transfinite extension of the $X_0, X_{i+1}/X_i$. Thus its colimit is in $\class{W}$ by Eklof's lemma (assuming $\lambda$ is a limit ordinal; if it is a successor ordinal $\lambda = \alpha +1$, then the colimit coincides with $X_{\alpha} \in \class{W}$).

Step two: We show that $\class{W}$ is closed under direct limits.  Again, it is enough to prove this for well-ordered \emph{continuous} direct limits. (See~\cite[Section~1.6]{adamek-rosicky}, especially Corollary~1.7 and the Remark that follows it where well-ordered direct limits are referred to as \emph{chains} and continuous well-ordered direct limits as \emph{smooth chains}.) So we consider a $\lambda$-diagram $$X_0 \xrightarrow{f_{0,1}} X_1 \xrightarrow{f_{1,2}} X_2 \xrightarrow{f_{2,3}} \cdots X_{i} \xrightarrow{f_{i,i+1}} X_{i+1} \xrightarrow{} \cdots$$ with each $X_i \in \class{W}$ and $X_{\gamma} = \varinjlim_{i < \gamma} \{X_i, f_{i,i+1}\}$ for each limit ordinal $\gamma < \lambda$. Our job is to show that
$\varinjlim_{i < \lambda} X_i \in \class{W}$. We will assume $\lambda$ is a limit ordinal, for otherwise the direct limit just equals $X_{\lambda - 1} \in \class{W}$. Following a standard way for defining direct limits (for example, see~\cite[Prop~IV.8.4, p.100]{stenstrom}), $\varinjlim_{i < \lambda} X_i $ is the cokernel of the following homomorphism: $\bigoplus_{i<j}X_{ij} \xrightarrow{} \bigoplus_{i<\lambda} X_i$ where the first direct sum is taken over all pairs $i < j < \lambda$, and $X_{ij} = X_i$ is just a copy of the domain of the map $X_{i} \xrightarrow{f_{ij}} X_j$, and the map is defined on the $i$th coordinate by $x_i \mapsto e_ix_i - e_jf_{ij}x_i$ where the $e_i$ denotes the canonical injection into the coproduct. In other words, the direct limit is $(\bigoplus_{i<\lambda} X_i)/K$ where $K$ is the image of this map: $$K = \ <e_ix_i - e_jf_{ij}x_i \ |\ \text{$x_i \in X_i$ and $i<j<\lambda$} >$$
Note that since the maps $e_i$ and $f_{ij}$ are linear we have that $K$ is the set of all finite sums of the form $e_ix_i - e_jf_{ij}x_i$ where the $x_i$ range through $X_i$ and $i<j$ ranges through all $i<j<\lambda$. There is a short exact sequence $$0 \xrightarrow{} K \xrightarrow{\subseteq} \bigoplus_{i < \lambda} X_i \xrightarrow{} \varinjlim_{i < \lambda} X_i \xrightarrow{} 0 , $$ and since $\class{W}$ is closed under direct sums and is thick, it is enough to show that $K \in \class{W}$. We now show that $K$ is a direct union (not well ordered, but still a direct union) of modules in $\class{W}$. So the proof will follow from the above Step~1.

Thinking of $\lambda$ as the set of all its smaller ordinals we define, for each finite subset $J \subseteq \lambda$ with $|J| >1$, the mapping
$$\phi_J : \bigoplus_{i\in J\backslash{\{j\}}} X_i \xrightarrow{} \bigoplus_{i<\lambda} X_i$$
where $j$ denotes the maximum element of the finite subset $J$, and the map is defined on the $i$th coordinate via $x_i \mapsto e_ix_i - e_jf_{ij}x_i$. One now verifies each of the following:
\begin{enumerate}
\item $\class{S} = \{J \subseteq \lambda \ | \ 1< |J| < \omega\}$ is a directed poset and there is a functor $D : \class{S} \xrightarrow{} R\text{-Mod}$ defined on objects by $J \mapsto
\bigoplus_{i\in J\backslash{\{j\}}} X_i$, and on arrows by taking an inclusion $J \subseteq J'$ to the map $D_{JJ'}$ defined on the $i$th coordinate as follows: $x_i \mapsto e_ix_i$ if $j = j'$ (that is, just a natural inclusion if $J$ and $J'$ have the same maximal element), but via $x_i \mapsto e_ix_i - e_jf_{ij}x_i$ if $j<j'$.

\item Each $D_{JJ'}$ is a monomorphism. In fact it is a split monomorphism with retraction map the canonical projection.

\item Each $\phi_J$ is also a split monomorphism with similar retraction. The image of $\phi_J$ identifies $\bigoplus_{i\in J\backslash{\{j\}}} X_i$ with the submodule
 $$K_J = \ <e_ix_i - e_jf_{ij}x_i \ |\ \text{$x_i \in X_i$ and $i \in J\backslash \{j\}$} >.$$ That is $K_J$ is the set of all finite sums of elements of the form $e_ix_i - e_jf_{ij}x_i$ as $i$ ranges through $J \backslash \{j\}$ and $x_i$ ranges through $X_i$.

 \item The direct system of monomorphisms $D : \class{S} \xrightarrow{} R\text{-Mod}$ is isomorphic via the natural transformation $\{\phi_J\}$ to the direct system of submodules $K_J$. The direct limit $\varinjlim_{i < \lambda} D$ identifies with the direct union of submodules $\bigcup_{J \in \class{S}} K_J \subseteq
\bigoplus_{i < \lambda} X_i$.

\item $K = \bigcup_{J \in \class{S}} K_J$.
\end{enumerate}

Now since each $\bigoplus_{i\in J\backslash{\{j\}}} X_i \in \class{W}$, we conclude from (3), (4), (5) and Step~1 that $K \in \class{W}$. This completes the proof.
\end{proof}

Of course generalizations of Proposition~\ref{prop-direct limits} to categories beyond $R$-Mod are possible and will be important for applications.  The reader will note that the above proof holds in any Grothendieck category, giving the following generalization.

\begin{proposition}\label{prop-Grothendieck direct limits}
Let $(\class{W},\class{F})$ be a cotorsion pair in a Grothendieck category $\cat{G}$ with $\class{W}$ thick. Then $\class{W}$ is closed under direct limits. In particular, the class $\class{W}$ of trivial objects is closed under direct limits whenever $(\class{W},\class{F})$ be an injective cotorsion pair.
\end{proposition}

Proposition~\ref{prop-Grothendieck direct limits} has some interesting consequences that apply to each of the injective cotorsion pairs recently appearing in~\cite{bravo-gillespie-hovey, gillespie-recollement}. For example, it is shown in~\cite{bravo-gillespie-hovey} that for any ring $R$ we have an injective cotorsion pair $(\class{W},\class{GI})$ in $R$-Mod where $\class{GI}$ is the class of Gorenstein AC-injective modules. The next corollary says that all modules of finite flat dimension are sent to zero in the corresponding stable homotopy category. Consequently, any Gorenstein AC-injective module is cotorsion and must be injective if it has finite flat dimension.

\begin{corollary}\label{cor-fin-flat-dim}
Let $(\class{W},\class{F})$ be an injective cotorsion pair in $R$-Mod or $\ch$. Then $\class{W}$ contains all objects of finite flat dimension. Consequently, any fibrant object is cotorsion and must be injective whenever it is of finite flat dimension.
\end{corollary}

\begin{proof}
Certainly $\class{W}$ contains all projectives and hence all flat objects since these are precisely direct limits of projectives. Since $\class{W}$ is thick and contains the flat objects it also contains all objects of finite flat dimension. The final statement follows since $\class{W} \cap \class{F}$ coincides with the class of injective objects.
\end{proof}

Recall that a chain complex is \emph{flat} if it is exact and each cycle module is a flat $R$-module. These complexes are categorically flat in $\ch$; they are direct limits of finitely generated projective complexes~\cite[Theorem~4.1.3]{garcia-rozas}. A complex $X$ is called \emph{DG-flat} if each $X_n$ is flat and $E \tensor X$ is exact whenever $E$ is an exact chain complex of (right) $R$-modules. $E \tensor X$ is the usual tensor product of chain complexes recalled in Section~\ref{subsec-complexes}. Again, the book~\cite{garcia-rozas} is a standard reference for DG-flat complexes.

\begin{corollary}\label{cor-flats}
Let $(\class{W},\class{F})$ be an injective cotorsion pair in $\ch$. Then each of the following hold:
\begin{enumerate}
\item Any exact complex $Y \in \class{F}$ has all cotorsion cycle modules. That is, each $Z_nY$, and hence each $Y_n$, is a cotorsion module.
\item If every complex in $\class{F}$ is exact, then
each DG-flat complex is in $\class{W}$.
\end{enumerate}
\end{corollary}

Here we will use the notation from~\cite{gillespie} where $\dgclass{C}$ denotes the class of all complexes $C$ with each $C_n$ cotorsion and with the property that all chain maps $F \xrightarrow{} C$ are null homotopic whenever $F$ is a flat chain complex. It turns out that $\dgclass{C}$ is precisely the right Ext orthogonal to the class of flat complexes. That is,   $\dgclass{C}$ is really the class of cotorsion objects in the category of complexes.

\begin{proof}
For (1), we have that $Y$ is cotorsion by Corollary~\ref{cor-fin-flat-dim}. That is, $Y \in \dgclass{C}$. The result now follows from~\cite[Theorem~3.12]{gillespie} which tells us that the exact complexes in $\dgclass{C}$ are precisely the ones with cotorsion cycle modules.

 For (2), in the case that every complex $Y \in \class{F}$ is exact we have $$\Ext^1_{\ch}(S^n(P),Y) \cong \Ext^1_R(P,Z_nY) = 0$$ for each projective module $P$. So $S^n(P) \in \class{W}$ for each projective $P$, and hence $S^n(F) \in \class{W}$ for each flat module $F$. But it follows from~\cite[Prop.~3.8]{gillespie-quasi-coherent} that each DG-flat complex is a retract of a transfinite extension of complexes of the form $S^n(F)$ with $F$ flat. Applying the well known Eklof's lemma, we conclude that all DG-flat complexes are in $\class{W}$, proving~(2).
\end{proof}

Going back to the Gorenstein AC-injective modules appearing before Corollary~\ref{cor-fin-flat-dim}, we get the following corollary.

\begin{corollary}\label{cor-cotorsions}
Any chain complex $X$ of Gorenstein AC-injective modules is in $\dgclass{C}$. Consequently, any exact complex of Gorenstein AC-injectives must have cotorsion cycle modules.
\end{corollary}

\begin{proof}
We refer to~\cite[Proposition~7.2]{gillespie-recollement}. It shows that there is an injective cotorsion pair in $\ch$ whose right side consists of the class of all complexes of Gorenstein AC-injective modules. The rest follows from what we already observed above.
\end{proof}

Note that Corollary~\ref{cor-cotorsions} implies Stovicek's original result that any exact complex of injectives must have cotorsion cycle modules. However, the reader familiar with~\cite{stovicek-purity} will realize that our methods are just a generalization of those employed by Stovicek.
These methods however don't dualize. So the following result is quite interesting. It is the dual of Stovicek's result, assuming $R$ is a coherent ring.

\begin{theorem}\label{them-projcycles}
Let $R$ be a coherent ring and $P$ an exact complex of projectives. Then $\Ext^1_R(Z_nP,A) = 0$ for any absolutely pure module $A$.  Consequently, any cycle $Z_nP$ must be a retract of a transfinite extension of finitely presented modules.
\end{theorem}

\begin{proof}
Since $P$ is exact, we have $\Ext^1_R(Z_nP,A) = \Ext^1_{\ch}(P,S^0(A))$, and so we can instead show that $\Ext^1_{\ch}(P,S^0(A)) = 0$ for all absolutely pure $A$. But it is enough to show $\homcomplex(P,S^0(A)) = \Hom_R(P,A)$ is exact whenever $P$ is an exact complex of projectives and $A$ is absolutely pure.  Let $\class{D}$ denote the class of all absolutely pure $R$-modules, and $\class{C}$ denote the class of all flat (right) $R$-modules. Since $R$ is coherent we have that $(\class{C},\class{D})$ forms a perfect duality pair with respect to the character modules of Section~\ref{subsec-character duality}. Therefore, by~\cite[Theorem~A.6]{bravo-gillespie-hovey}, $\Hom_R(P,A)$ is exact for all absolutely pure $A$ if and only if $F \otimes_R P$ is exact for all flat $F$. But certainly this is true.

The last statement holds since the ``set'' $\class{S}$, of all finitely presented modules, cogenerates the complete cotorsion pair $(\leftperp{(\rightperp{\class{S}})},\rightperp{\class{S}})$ and $\rightperp{\class{S}}$ is precisely the class of all absolutely pure modules.
\end{proof}

\begin{remark}
Suppose $R$ is any ring in which all level modules have finite flat dimension. Then the above argument will extend to show $\Ext^1_R(Z_nP,A) = 0$ for all absolutely clean modules $A$ and exact complexes of projectives $P$.
\end{remark}

\section{Applications to Ding injective modules and complexes}\label{sec-Ding injective}

This section is devoted to studying Ding injective modules and chain complexes over Ding-Chen rings.
As in~\cite{gillespie-ding}, a \emph{Ding-Chen ring} is a (left and right) coherent ring with ${}_RR$ and $R_R$ each of finite FP-injective (absolutely pure) dimension. We refer to~\cite{gillespie-ding} for more on Ding-Chen rings and Ding injective modules although the definitions below should suffice for our purposes.

\begin{definition}\label{def-Ding injectives}
We call an $R$-module $M$ \textbf{Ding injective} if there exists an exact complex of injectives $$\cdots \rightarrow I_1 \rightarrow I_0 \rightarrow I^0 \rightarrow I^1 \rightarrow \cdots$$ with $M = \ker{(I^0 \rightarrow I^1)}$ and which remains exact after applying $\Hom_R(A,-)$ for any absolutely pure module $A$.

In the same way, we call a chain complex $X$ \emph{Ding injective} if there exists an exact complex of injective complexes $$\cdots \rightarrow I_1 \rightarrow I_0 \rightarrow I^0 \rightarrow I^1 \rightarrow \cdots$$ with $X = \ker{(I^0 \rightarrow I^1)}$ and which remains exact after applying $\Hom_{\ch}(A,-)$ for any absolutely pure chain complex $A$. Recall that a chain complex $I$ is injective (resp. absolutely pure) in $\ch$ if and only if it is exact and each cycle $Z_nI$ is an injective (resp. absolutely pure) module.
\end{definition}

The content of Corollaries~\ref{cor-Stov-injective-modules} and~\ref{cor-Ding-injective-complexes} is that the $\Hom$ condition in the above definitions come automatically when $R$ is a Ding-Chen ring.
In fact, for the module case, we have the following stronger statement.

\begin{theorem}\label{them-D-injective-modules}
Let $R$ be a Ding-Chen ring. Then  $M$ is Ding injective if and only if $M = Z_0I$ for some exact complex $I$ of Ding injective $R$-modules. That is, any exact complex of Ding injectives automatically has Ding injective cycles.
\end{theorem}

\begin{proof}
Since $R$ is Ding-Chen, the modules of finite flat dimension coincide with the modules of finite absolutely pure dimension~\cite{ding and chen 93}. Denote this class by $\class{W}$. We have that $(\class{W},\class{DI})$ is an injective cotorsion pair, where $\class{DI}$ is the class of all Ding injective modules; see~\cite[Theorem~3.4]{ding and mao 07} and~\cite[Corollary~4.5]{gillespie-ding}. By~\cite[Prop.~7.2]{gillespie-recollement} it lifts to another injective cotorsion pair $(\leftperp{\exclass{DI}},\exclass{DI})$ on $\ch$, where $\exclass{DI}$ is the class of all exact complexes of Ding injectives. By Corollary~\ref{cor-flats}~(2) we have that $S^n(F) \in \leftperp{\exclass{DI}}$ for each flat module $F$ and hence for each $F$ of finite flat dimension. Thus for all $W \in \class{W}$ and $I \in \exclass{DI}$ we have $0 = \Ext^1(S^n(W),I) = \Ext^1_R(W,Z_nI)$. Hence each $Z_nI$ must be Ding injective.
\end{proof}

Since injective modules are Ding injective we deduce the following corollary.

\begin{corollary}\cite[Prop.~7.9]{stovicek-purity}\label{cor-Stov-injective-modules}
Let $R$ be a Ding-Chen ring. Then $M$ is Ding injective if and only if $M = Z_0I$ for some exact complex $I$ of injective $R$-modules. Therefore,  for any absolutely pure $A$, any exact complex $I$ of injectives will remain exact after applying $\Hom_R(A,-)$.
\end{corollary}


It was shown by Yang, Liu, and Liang in~\cite{Ding-Chen-complex-models} that the Ding injective complexes are precisely the complexes $X$ for which each $X_n$ is Ding injective and all chain maps $A \xrightarrow{} X$ are null homotopic whenever $A$ is an absolutely pure (FP-injective) chain complex. We now show that the null homotopy condition is automatic when $R$ is a Ding-Chen ring.

\begin{corollary}\label{cor-Ding-injective-complexes}
Let $R$ be a Ding-Chen ring. Then a complex $X$ is Ding injective if and only if each $X_n$ is a Ding injective $R$-module.
\end{corollary}

\begin{proof}
The ``only if'' part is easy from the definition of Ding injective. For the converse, we use that for any coherent ring $R$, the Ding injective cotorsion pair $(\class{W},\class{DI})$ on $R$-Mod lifts to an injective model structure $(\leftperp{\dwclass{DI}}, \dwclass{DI})$ on $\ch$, by~\cite[Proposition~7.2]{gillespie-recollement}.  Here, $\dwclass{DI}$ is the class of all complexes of Ding injective modules.
From Corollary~\ref{cor-fin-flat-dim} we have that each complex of finite flat dimension is in $\leftperp{\dwclass{DI}}$. But in the case that $R$ is Ding-Chen, it is easy to see that any absolutely pure complex $A$ has finite flat dimension. Indeed such an $A$ is an exact complex with each cycle module $Z_nA$ absolutely pure. Thus each $Z_nA$ has finite flat dimension. Now taking a flat resolution of $A$, the upper bound on the flat dimensions implies that $A$ itself has finite flat dimension. Hence $A  \in \leftperp{\dwclass{DI}}$ for each absolutely pure $A$.
It follows that $\homcomplex(A,X)$ is exact for any absolutely pure $A$ and any complex $X \in \dwclass{DI}$, whenever $R$ is Ding-Chen.
\end{proof}

\begin{remark}
It was shown by Yang, Liu, and Liang in~\cite{Ding-Chen-complex-models} that the class of Ding injective complexes is precisely $\dgclass{DI}$ whenever $R$ is Ding-Chen. So we now have shown that both  $\dgclass{DI} = \dwclass{DI}$ and $\exclass{DI} = \tilclass{DI}$, whenever $R$ is a Ding-Chen ring.  (Again, we are using the notation from~\cite{gillespie}.)
\end{remark}

\section{Complexes of projective complexes}\label{sec-complexes of projective complexes}

Our main purpose here is to prove Theorem~\ref{thm-dual-exact}, which will be our tool in Section~\ref{sec-Ding projectives} for characterizing Ding projective complexes. We get this by generalizing results in~\cite[Appendix~A]{bravo-gillespie-hovey}, from modules to chain complexes.

Since $\ch$ is itself an abelian category we can of course consider  $\textnormal{Ch}(\ch)$, the category of chain complexes of chain complexes. Using~\cite[Sign Trick~1.2.5]{weibel}, the category $\textnormal{Ch}(\ch)$ can be identified with the category of bicomplexes. However, for our purpose here it is somewhat easier to stick with the category $\textnormal{Ch}(\ch)$.

\subsection{Free chain complexes} A free $R$-module is one that is isomorphic to a direct sum of copies of $R$. Analogously, we say a chain complex $F \in \ch$ is \textbf{free} if it is isomorphic to a direct sum $\oplus_{i \in I}D^{n_i}(R)$ where each $n_i$ is some integer. Clearly this is equivalent to saying $F$ is isomorphic to $\oplus_{n \in \Z} D^n(F_n)$ where each $F_n$ is some free $R$-module. It is also equivalent to define $F$ to be an exact complex with each $Z_nF$ a free $R$-module. However, it will be sufficient and most convenient for us to use the representations $\oplus_{i \in I}D^{n_i}(R)$.

Evidently free complexes are closed under arbitrary direct sums. It is also easy to check that free complexes are projective objects in $\ch$, and that each projective complex $P$ is a retract of a free complex.

\begin{lemma}[Eilenberg's swindle]\label{lemma-eilenberg-swindle}
Given any projective chain complex $P$, there exists a free chain complex $F$ such that $P \oplus F \cong F$.
\end{lemma}

\begin{proof}
We follow~\cite[Corollary~2.7]{lam}. Since $P$ is projective, as noted above we can find another (projective) complex $Q$ such that $P \oplus Q$ is a free complex. Setting
$$F = (P \oplus Q) \oplus (P \oplus Q) \oplus (P \oplus Q) \oplus \cdots$$ will produce a free complex as desired. Indeed $$F \cong P \oplus (Q \oplus P) \oplus (Q \oplus P) \oplus (Q \oplus P) \cdots \cong P \oplus F.$$
\end{proof}

We now need to study chain complexes of projective chain complexes. That is, objects of $\textnormal{Ch}(\ch)$ with each component a projective complex.

\begin{lemma}\label{lemma-complexes of projectives are retracts of complexes of frees}
Let $R$ be any ring, and $\mathbb{P} \in \textnormal{Ch}(\ch)$ a complex of projective chain complexes. Then
$\mathbb{P}$ is a direct summand of a complex $\mathbb{F}$ of free chain complexes. Furthermore, if $\mathbb{P}$
is exact then $\mathbb{F}$ can be taken to be exact.
\end{lemma}

\begin{proof}
By Lemma~\ref{lemma-eilenberg-swindle} we may, for each $\mathbb{P}_n$, find a free chain complex $F^n$ such that $\mathbb{P}_n \oplus F^n \cong F^n$. We form $\mathbb{F} = \mathbb{P} \oplus (\oplus_{n \in \Z} D^n(F^n))$. This is a complex of complexes which in degree $n$ is the chain complex $\mathbb{P}_n \oplus
F^n \oplus F^{n+1} \cong F^n \oplus F^{n+1}$, and so a free chain complex.
We are done since $\mathbb{P}$ is a direct summand of $\mathbb{P} \oplus (\oplus_{n \in \Z} D^n(F^n))$. Moreover, $\mathbb{P} \oplus (\oplus_{n \in \Z} D^n(F^n))$ is exact
whenever $\mathbb{P}$ is exact.
\end{proof}

\begin{theorem}\label{thm-bounded above complexes of finitely generated frees cogenerate}
Let $R$ be any ring, and let $\class{S}$ be the set of all bounded above complexes of finitely generated free complexes. The cotorsion pair $(\leftperp{(\rightperp{S})}, \rightperp{S})$ in $\textnormal{Ch}(\ch)$, cogenerated by $\class{S}$, is functorially complete. Moreover, $\class{C} = \leftperp{(\rightperp{S})}$ is the class of all complexes of projective complexes.
\end{theorem}

\begin{proof}
Since $\textnormal{Ch}(\ch)$ is a Grothendieck category with enough projectives it follows from~\cite[Corollary~6.8]{hovey} that $\class{S}$ cogenerates a functorially complete cotorsion pair
$(\leftperp{(\rightperp{S})}, \rightperp{S})$ in $\textnormal{Ch}(\ch)$. Setting $\class{C} = \leftperp{(\rightperp{S})}$, our aim is to show that this is precisely the class of all complexes of projective complexes. We note that $S$ contains a set of projective generators for $\textnormal{Ch}(\ch)$; for example, the set $\{D^n(D^m(R))\}_{n,m\in \Z}$. So it follows from~\cite[Corollary~6.9/Proof of~6.5]{hovey} that
$\leftperp{(\rightperp{S})}$ is precisely the class of all
direct summands of transfinite extensions of objects in
$S$. By Lemma~\ref{lemma-complexes of projectives are retracts of complexes of
frees} we only need to show that a complex of free
complexes is a transfinite extension of bounded above complexes of
finitely generated free complexes.  So let $\mathbb{F}$ be a complex of free
complexes and write each $\mathbb{F}_n = \oplus_{i \in I_n}D^{n_i}(R_i)$ for some indexing set $I_n$ with each $R_i = R$. We assume that $\mathbb{F}$ is not (isomorphic to) a complex in $\class{S}$, so we can certainly find
a nonzero $\mathbb{F}_n$, and we then take just one summand $D^{n_j}(R_j)$ for some $j \in
I_n$. We start to build a bounded above subcomplex $\mathbb{X} \subseteq \mathbb{F}$ by
setting $\mathbb{X}_n = D^{n_j}(R_j)$ and setting $\mathbb{X}_i = 0$ for all $i > n$.
Now note $d(D^{n_j}(R_j)) \subseteq \oplus_{i \in I_{n-1}}  D^{n_i}(R_i)$ and that we can find a \emph{finite}
subset $L_{n-1} \subseteq I_{n-1}$ such that $d(D^{n_j}(R_j)) \subseteq \oplus_{i \in L_{n-1}}
 D^{n_i}(R_i)$. We set $\mathbb{X}_{n-1} = \oplus_{i \in L_{n-1}} D^{n_i}(R_i)$.
We can continue down in the same way
finding $L_{n-2} \subseteq I_{n-2}$ with $|L_{n-2}|$ finite and with
$d(\oplus_{i \in L_{n-1}}D^{n_i}(R_i)) \subseteq \oplus_{i \in L_{n-2}}
D^{n_i}(R_i)$.
In this way we can construct the subcomplex $$\mathbb{X} = \cdots
\xrightarrow{} 0 \xrightarrow{} D^{n_j}(R_j) \xrightarrow{} \oplus_{i \in
L_{n-1}} D^{n_i}(R_i) \xrightarrow{} \oplus_{i \in L_{n-2}} D^{n_i}(R_i) \xrightarrow{}
\cdots $$
Note that $\mathbb{X}$ is a nonzero bounded above complex of finitely
generated free complexes. That is, $\mathbb{X} \in \class{S}$.

We set $\mathbb{X}^0 = \mathbb{X}$. To finish the proof we will argue that we can write $\mathbb{F}$ as the union
of a continuous chain $0 \neq \mathbb{X}^0 \subsetneq \mathbb{X}^1 \subsetneq \cdots \subsetneq \mathbb{X}^{\alpha} \subsetneq \mathbb{X}^{\alpha +1} \subsetneq \cdots $ with each $\mathbb{X}^{\alpha +1}/\mathbb{X}^{\alpha} \in \class{S}$. So we consider $\mathbb{F}/\mathbb{X}^0$, and in the same way as above we find a bounded above subcomplex $\mathbb{X}^1/\mathbb{X}^0 \subseteq \mathbb{F}/\mathbb{X}^0$ consisting of finitely generated free complexes in each degree. However, we are careful to construct $\mathbb{X}^1$ as follows: Let $L^0_{n}$ denote the indexing sets of the previously constructed complex $\mathbb{X}^0_n = \oplus_{i \in L^0_{n}} D^{n_i}(R_i)$.
Note that we can
identify the quotient $\mathbb{F}/\mathbb{X}^0$ with a complex whose degree $n$
entry is $\oplus_{i \in I_n-L^0_n} D^{n_i}(R_i)$. And so we may take $\mathbb{X}^1$ so that $\mathbb{X}^1_n = \oplus_{i \in L^1_{n}} D^{n_i}(R_i)$ with each $L^1_n$ finite and $L^0_n \subseteq L^1_n \subseteq I_n$ for each $n$.
Continuing this process, we
continue to construct an increasing union
$0 \neq \mathbb{X}^0 \subsetneq \mathbb{X}^1 \subsetneq \mathbb{X}^2 \subsetneq \cdots $ corresponding to a nested union of
subsets $L^0_n \subseteq L^1_n \subseteq L^2_n \subseteq \cdots$ for
each $n$.
Assuming this process doesn't terminate we set $\mathbb{X}^{\omega} =
\cup_{\alpha < \omega} \mathbb{X}^{\alpha}$ and note that $\mathbb{X}^{\omega}_n =
\oplus_{i \in L^{\omega}_n} D^{n_i}(R_i)$ where $L^{\omega}_n =
\cup_{\alpha < \omega} L^{\alpha}_n$.
Of course $\mathbb{X}^{\omega}$ and
$\mathbb{F}/\mathbb{X}^{\omega}$ are complexes of free complexes and continuing this process we can
obtain an ordinal $\lambda$ and a continuous chain
$0 \neq \mathbb{X}^0 \subsetneq \mathbb{X}^1 \subsetneq \cdots \subsetneq \mathbb{X}^{\alpha} \subsetneq \mathbb{X}^{\alpha +1} \subsetneq \cdots $
with $\mathbb{F} = \cup_{\alpha <
\lambda} \mathbb{X}^{\alpha}$ and $\mathbb{X}^0, \mathbb{X}^{\alpha +1}/\mathbb{X}^{\alpha} \in \class{S}$.
\end{proof}

\subsection{Pure exact complexes of complexes} Let $\class{E} : 0 \xrightarrow{} X \xrightarrow{} Y \xrightarrow{} Z \xrightarrow{} 0$ be a short exact sequence in $\ch$. We say that $\class{E}$ is \emph{pure} if for each finitely presented complex $F$, the sequence of abelian groups $\Hom_{\ch}(F,\class{E})$, that is $$0 \xrightarrow{} \Hom_{\ch}(F,X) \xrightarrow{} \Hom_{\ch}(F,Y) \xrightarrow{} \Hom_{\ch}(F,Z) \xrightarrow{} 0,$$ is also exact. Pure exact sequences of complexes can be characterized in terms of the functors $\overline{\homcomplex}$ and $\overline{\otimes}$ from  Section~\ref{subsec-modified hom and tensor}. In particular, $\class{E}$ is pure exact if and only if $\overline{\homcomplex}(F,\class{E})$ is a short exact sequence of complexes for each finitely presented complex $F$ if and only if $F \overline{\otimes} \class{E}$ is a short exact sequence of complexes for each finitely presented complex $F$ (or just \emph{any} complex $F$). See~\cite[Theorem~5.1.3]{garcia-rozas}. This leads us to the following notion of a pure exact complex of complexes.

\begin{definition}\label{def-pure exact complexes}
Let $\mathbb{X} \in \textnormal{Ch}(\ch)$. We say that $\mathbb{X}$ is a \textbf{pure exact complex of complexes} if the following equivalent conditions are satisfied:
\begin{enumerate}
\item $\Hom_{\ch}(F,\mathbb{X}) = \homcomplex(S^0(F),\mathbb{X})$ is an exact complex of abelian groups for each finitely presented chain complex $F$.
\item $\overline{\homcomplex}(F,\mathbb{X})$  is an exact complex (of complexes of abelian groups) for each finitely presented chain complex $F$.
\item $F \overline{\otimes} \mathbb{X}$ is an exact complex (of complexes of abelian groups) for each finitely presented chain complex $F$.
\item $F \overline{\otimes} \mathbb{X}$ is an exact complex (of complexes of abelian groups) for each chain complex $F$.
\end{enumerate}
\end{definition}

\begin{lemma}\label{lem-pure}
Suppose $\mathbb{X}$ is a bounded complex of finitely
presented complexes and $\mathbb{Y}$ is a pure exact complex of complexes.  Then every chain
map $f\mathcolon \mathbb{X} \xrightarrow{} \mathbb{Y}$ is chain homotopic to $0$.
\end{lemma}

\begin{proof}
Note that it is equivalent to show $\homcomplex(\mathbb{X},\mathbb{Y})$ is exact, where $\homcomplex$ is the usual hom-complex of Section~\ref{subsec-complexes} (but applied to the case of $\cat{A} = \ch$ instead of the more typical application of $\cat{A} = \rmod$).

Let $m$ be the largest degree $i$ for which
$\mathbb{X}_{i}$ is nonzero, and let $\mathbb{A}$ be the subcomplex of $\mathbb{X}$ with
$\mathbb{A}_{i}=\mathbb{X}_{i}$ for $i<m$ and $\mathbb{A}_{m}=0$.  We have a degreewise split short exact sequence
\[
0 \xrightarrow{} \mathbb{A} \xrightarrow{} \mathbb{X} \xrightarrow{}  S^{m}(\mathbb{X}_{m})
\xrightarrow{} 0
\]
which for the given $\mathbb{Y}$ induces the short exact sequence
\[
0 \xrightarrow{}\homcomplex (S^{m}(\mathbb{X}_{m}),\mathbb{Y}) \xrightarrow{} \homcomplex (\mathbb{X},\mathbb{Y})
\xrightarrow{} \homcomplex (\mathbb{A},\mathbb{Y}) \xrightarrow{} 0.
\]
Now $\homcomplex (S^{m}(\mathbb{X}_{m}),\mathbb{Y})$ is exact by condition (1) of Definition~\ref{def-pure exact complexes}. The long exact sequence in homology now gives us the result by induction on $m$.
\end{proof}

Let $(\class{C},\class{W})$ denote the cotorsion pair of Theorem~\ref{thm-bounded above complexes of finitely generated frees cogenerate}. The result below tells us that pure exact complexes of complexes are in $\class{W}$.

\begin{theorem}\label{thm-pure}
Let $R$ be any ring and let $\mathbb{C}$ be a complex of projective complexes.
If $\mathbb{Y}$ is a pure exact complex of complexes, then $\homcomplex (\mathbb{C},\mathbb{Y})$
is exact, or, equivalently $\mathbb{Y} \in \class{W} = \rightperp{\class{C}}$.
\end{theorem}

\begin{proof}
 In view of Theorem~\ref{thm-bounded
above complexes of finitely generated frees cogenerate}, it suffices to assume
that $\mathbb{C}$ is a bounded above complex of finitely generated free
complexes and  to show that any chain map $f\mathcolon
\mathbb{C} \xrightarrow{} \mathbb{Y}$ is chain homotopic to $0$.  We construct a chain
homotopy $D_{n}\mathcolon \mathbb{C}_{n}\xrightarrow{} \mathbb{Y}_{n+1}$ with
$dD_{n}+D_{n-1}d=f_{n}$ by downwards induction on $n$.  Since $\mathbb{C}$ is
bounded above, we can take $D_{n}=0$ for large $n$ to begin the
induction.  So we suppose that $D_{i}$ has been defined for $i\geq n$
and that $dD_{n+1}+D_{n}d=f_{n+1}$.

The idea is to replace $D_{n}$ with a new map $\widetilde{D}_{n}$ still having $dD_{n+1}+\widetilde{D}_{n}d=f_{n+1}$, and to also find a map $D_{n-1}$ such that $d\widetilde{D}_{n} +D_{n-1}d=f_{n}$. First note that
\[
(f_{n}-dD_{n})d = d (f_{n+1}-D_{n}d)=d^{2}D_{n+1}=0,
\]
so there is an induced map $g_{n}\mathcolon
\mathbb{C}_{n}/B_{n}\mathbb{C}\xrightarrow{}\mathbb{Y}_{n}$, such that the composite $\mathbb{C}_n \xrightarrow{\pi} \mathbb{C}_{n}/B_{n}\mathbb{C}\xrightarrow{g_n} \mathbb{Y}_{n}$ equals $f_n - dD_n$.  Now consider the bounded complex
$\mathbb{X}$ of finitely presented complexes with $\mathbb{X}_{n}=\mathbb{C}_{n}/B_{n}\mathbb{C}$,
$\mathbb{X}_{n-1}=\mathbb{C}_{n-1}$, $\mathbb{X}_{n-2}=\mathbb{C}_{n-1}/B_{n-1}\mathbb{C}$, and $\mathbb{X}_{i}=0$ for all
other $i$.  There is a chain map $g\mathcolon \mathbb{X} \xrightarrow{} \mathbb{Y}$ that
is $g_{n}$ in degree $n$, $f_{n-1}$ in degree $n-1$, and $f_{n-2}\bar{d}$ in
degree $n-2$.  By Lemma~\ref{lem-pure}, this chain map must be chain
homotopic to $0$.  This gives us maps $D_{n}'\mathcolon
\mathbb{C}_{n}/B_{n}\mathbb{C}\xrightarrow{}\mathbb{Y}_{n+1}$ and $D_{n-1}\mathcolon
\mathbb{C}_{n-1}\xrightarrow{}\mathbb{Y}_{n}$ such that
$dD_{n}'+D_{n-1}\bar{d} = g_n$. Upon composing with $\mathbb{C}_n \xrightarrow{\pi} \mathbb{C}_{n}/B_{n}\mathbb{C}$ this becomes $d D'_n \pi + D_{n-1}d = f_{n}-dD_{n}$.
Now setting $\widetilde{D}_{n} = D_n + D_{n}'\pi$,
we still have the required relation $dD_{n+1} + \widetilde{D}_{n}d=f_{n+1}$ because $dD_{n+1} + ( D_n + D_{n}'\pi)d = dD_{n+1} + D_nd + 0 = f_{n+1}$.
Moreover, we have the other promised relation: $d \widetilde{D}_{n}+D_{n-1}d =f_{n}$. Indeed $d \widetilde{D}_{n}+D_{n-1}d = d(D_n + D'_n \pi) + D_{n-1}d =dD_n + (dD'_n \pi + D_{n-1}d) =  dD_n + (f_{n}-dD_{n}) = f_n$.
\end{proof}

Recall that for a complex $X$ we have its character dual $X^+ = \Hom_{\Z}(X,\Q)$. See Section~\ref{subsec-character duality}.

\begin{lemma}\label{lemma-dual-exact}
Let $\mathbb{C}$ be a complex of complexes, and $X$ a complex of right $R$-modules.
Then $X \overline{\otimes} \mathbb{C}$ is exact if and only if $\overline{\homcomplex}(\mathbb{C},X^+)$ is
exact.
\end{lemma}

\begin{proof}
$X \overline{\otimes} \mathbb{C}$ is exact if and only if
$(X \overline{\otimes} \mathbb{C})^{+}$ is exact. But also using parts (1) and (6)~b) of~\cite[Proposition~4.2.1]{garcia-rozas} we get
\[
(X \overline{\otimes} \mathbb{C})^{+} =  \Hom_{\Z}(X \overline{\otimes} \mathbb{C},\Q) = \overline{\homcomplex}(X \overline{\otimes} \mathbb{C},D^1(\Q))
\]
\[
\cong  \overline{\homcomplex}(\mathbb{C}, \overline{\homcomplex}(X,D^1(\Q))) \cong \overline{\homcomplex}(\mathbb{C},X^+).
\]
\end{proof}

We need just one more lemma before proving the main theorem.

\begin{lemma}\label{lemma-crazy-complexes}
Let $Y \in \ch$ be a complex, and $\mathbb{C} \in \textnormal{Ch}(\ch)$ be a complex of complexes . Then $\overline{\homcomplex}(\mathbb{C},Y)$ is exact if and only if $\homcomplex(\mathbb{C},S^n(Y))$ is exact for each $n$.
\end{lemma}

\begin{proof}
Pondering definitions, we see $\overline{\homcomplex}(\mathbb{C},Y)$ is the cochain complex, of complexes of abelian groups,   whose degree zero entry is the complex of abelian groups
$$\cdots \xleftarrow{}   \Hom_{\ch}(\Sigma^{n-1}\mathbb{C}_0,Y)    \xleftarrow{}    \Hom_{\ch}(\Sigma^n \mathbb{C}_0,Y) \xleftarrow{} \cdots$$
(where $\Hom_{\ch}(\Sigma^n \mathbb{C}_0,Y)$ is the degree $n$ component of this complex). So to say that the overall complex is exact means that the cochain complex of abelian groups, $\Hom_{\ch}(\Sigma^n \mathbb{C},Y)$, is exact for each $n$.

On the other hand, by again following definitions, $\homcomplex(\mathbb{C},S^n(Y))$ is isomorphic to the cochain complex, $\Hom_{\ch}(\Sigma^{-n} \mathbb{C},Y)$, of abelian groups.
So exactness of $\homcomplex(\mathbb{C},S^n(Y))$ for all $n$ is equivalent to exactness of $\overline{\homcomplex}(\mathbb{C},Y)$.
\end{proof}

Now if $\cat{C}$ is a collection of chain complexes of right $R$-modules, and
$\cat{D}$ is a collection of chain complexes of left $R$-modules, we say that
$(\cat{C},\cat{D})$ is a \textbf{duality pair} if $X \in \cat{C}$ if
and only if ${X}^{+} \in  \cat{D}$, and $Y \in \cat{D}$ if and
only if $Y^{+} \in \cat{C}$.  It is immediate from Corollary~\ref{cor-duality} that the absolutely clean and level complexes give rise to two duality pairs. One where $\class{C}$ is the class of all absolutely clean complexes of right $R$-modules, and another where $\class{C}$ is the class of all level complexes of right $R$-modules.

\begin{theorem}\label{thm-dual-exact}
Suppose $(\cat{C},\cat{D})$ is a duality pair in $\ch$
such that $\cat{D}$ is closed under pure quotients.  Let $\mathbb{C}$ be a
complex of projective complexes.  Then $X \overline{\otimes} \mathbb{C}$ is exact for all $X \in
\cat{C}$ if and only if $\overline{\homcomplex}(\mathbb{C},Y)$ is exact for all $Y \in
\cat{D}$.  In particular,
\begin{enumerate}
\item $A \overline{\otimes} \mathbb{C}$ is exact for all absolutely clean complexes $A$ if and only if $\overline{\homcomplex}(\mathbb{C},L)$ is exact for all level complexes $L$.
\item $L \overline{\otimes} \mathbb{C}$ is exact for all level complexes $L$ if and only if $\overline{\homcomplex}(\mathbb{C},A)$ is exact for all absolutely clean complexes $A$.
\end{enumerate}
\end{theorem}

\begin{proof}
In view of Lemma~\ref{lemma-dual-exact}, if $\overline{\homcomplex}(\mathbb{C},Y)$ is
exact for all $Y \in \cat{D}$, then $X \overline{\otimes} \mathbb{C}$ is exact for all
$X \in \cat{C}$.  Conversely, suppose $X \overline{\otimes} \mathbb{C}$ is exact for all
$X \in \cat{C}$.  Then if $Y \in \cat{D}$, $Y^{+}\overline{\otimes} \mathbb{C}$ is exact,
and so Lemma~\ref{lemma-dual-exact} tells us that $\overline{\homcomplex}(\mathbb{C},Y^{++})$ is exact.  We conclude that $\overline{\homcomplex}(\mathbb{C},D)$ is exact for
all $D \in \cat{D}^{++}$, and we note that $\cat{D}^{++}\subseteq
\cat{D}$ since $(\cat{C},\cat{D})$ is a duality pair.

Now, for any complex $Y$, the natural map $Y \xrightarrow{}Y^{++}$ is a pure
monomorphism of complexes~\cite[Proposition~5.1.4.(4)]{garcia-rozas}.  So if $Y \in
\cat{D}$, the quotient $Y^{++}/Y$ is also in $\cat{D}$ since $\cat{D}$
is closed under pure quotients.  We can therefore create a resolution
of $Y\in \cat{D}$ by elements of $\cat{D}^{++}$.  That is, we can
find a pure exact resolution of $Y$ by complexes of complexes in $\cat{D}^{++}$.
This gives us a short exact sequence
\[
0 \xrightarrow{} S^{n}(Y) \xrightarrow{} \mathbb{D} \xrightarrow{} \mathbb{P}
\xrightarrow{} 0
\]
in which $\mathbb{P}$ is pure exact and $\mathbb{D}$ is a bounded above complex with
entries in $\cat{D}^{++}$.  Theorem~\ref{thm-pure} tells us that
$\homcomplex (\mathbb{C},\mathbb{P})$ is exact.  Hence $\homcomplex (\mathbb{C},\mathbb{D}) \cong \homcomplex (\mathbb{C},S^n(Y))$. So by Lemma~\ref{lemma-crazy-complexes}, it only remains to show that  $\homcomplex (\mathbb{C},\mathbb{D})$ is exact whenever $\mathbb{D}$ is bounded above with components in $\class{D}^{++}$.

But showing $\homcomplex (\mathbb{C},\mathbb{X})$ is exact for all complexes of projective complexes $\mathbb{C}$, is equivalent to showing $\Ext^1(\mathbb{C},\mathbb{X}) = 0$ for all such $\mathbb{C}$. By Lemma~\ref{lemma-crazy-complexes} and the first paragraph we conclude $\Ext^1(\mathbb{C},S^n(D)) = 0$ whenever $D \in \cat{D}^{++}$. Now any bounded above complex $\mathbb{D}$ with entries in $\cat{D}^{++}$ can be seen to be an inverse transfinite extension of spheres $S^n(D)$ for some $D \in \cat{D}^{++}$. By the dual of the Eklof Lemma, we know that $\Ext^1(\mathbb{C},-)$ is closed under inverse transfinite extensions~\cite[Theorem~1.6]{enochs-iacob-jenda}. So $\Ext^1(\mathbb{C},\mathbb{D}) = 0$, whence $\homcomplex (\mathbb{C},\mathbb{D})$ is exact  for all complexes of projective complexes $\mathbb{C}$. This completes the proof.

As noted above, the absolutely clean complexes and level complexes give rise to two duality pairs. Moreover, each class is closed under pure quotients by~\cite[Propositions~2.7/4.7]{bravo-gillespie}.
\end{proof}


\section{Applications to Ding projective modules and complexes}\label{sec-Ding projectives}

We now wish to prove the duals to Corollaries~\ref{cor-Stov-injective-modules} and~\ref{cor-Ding-injective-complexes}. They are Theorem~\ref{them-Ding-projective-modules} and Corollary~\ref{cor-Ding-projective-complexes} below.

\begin{definition}\label{def-Ding projectives}
We call an $R$-module $M$ \textbf{Ding projective} if there exists an exact complex of projectives $$\cdots \rightarrow P_1 \rightarrow P_0 \rightarrow P^0 \rightarrow P^1 \rightarrow \cdots$$ with $M = \ker{(P^0 \rightarrow P^1)}$ and which remains exact after applying $\Hom_R(-,F)$ for any flat module $F$.

In the same way, we call a chain complex $X$ \emph{Ding projective} if there exists an exact complex of projective complexes $$\cdots \rightarrow P_1 \rightarrow P_0 \rightarrow P^0 \rightarrow P^1 \rightarrow \cdots$$ with $X = \ker{(P^0 \rightarrow P^1)}$ and which remains exact after applying $\Hom_{\ch}(-,F)$ for any flat chain complex $F$. Recall that a chain complex $P$ is projective (resp. flat) in $\ch$ if and only if it is exact and each cycle $Z_nP$ is a projective (resp. flat) module.
\end{definition}

\begin{theorem}\label{them-Ding-projective-modules}
Let $R$ be a Ding-Chen ring. Then a module $M$ is Ding projective if and only if $M = Z_0P$ for some exact complex $P$ of projective $R$-modules. In the same way, a chain complex $X$ is Ding projective if and only if $X = Z_0\mathbb{P}$ for some exact complex $\mathbb{P}$ of projective complexes.
\end{theorem}

\begin{proof} First, we look at the $R$-module case. Let $P$ be any exact complex of projectives. We need to show that $\Hom_R(P,F)$ remains exact for any flat (left) module $F$. But by~\cite[Theorem~6.7]{bravo-gillespie-hovey}, and since $R$ is coherent, it is enough to show $A \tensor_R P$ is exact for all absolutely pure (right) modules $A$.
Note that clearly $F \tensor_R P$ is exact for any flat (right) module $F$. It follows that $M \tensor_R P$ is exact for all $M$ of finite flat dimension. But since $R$ is a Ding-Chen ring, any absolutely pure module $A$ does have finite flat dimension.  So we are done. We note that a similar but alternate proof could be given, using Theorem~\ref{them-projcycles} instead.

Next, we see that the proof holds for chain complexes due to the work of Section~\ref{sec-complexes of projective complexes}. So now let $\mathbb{P}$ be any exact complex of projective complexes. We need to show that $\Hom_{\ch}(P,F)$ remains exact for any flat (left) chain complex $F$. However, looking at the definition of $\overline{\homcomplex}$ it is clear that this is equivalent to the statement that $\overline{\homcomplex}(\mathbb{P},F)$ remains exact for any flat complex $F$.
Now by Theorem~\ref{thm-dual-exact}, and since $R$ is coherent, this is equivalent to the statement that $A \overline{\otimes} \mathbb{P}$ is exact for all absolutely pure (right) complexes $A$.
We note that $F \overline{\otimes} \mathbb{P}$ is exact for any flat (right) complex $F$ by~\cite[Proposition~5.1.2]{garcia-rozas}. It follows that $X \overline{\otimes} \mathbb{P}$ is exact for all complexes $X$ of finite flat dimension. But since $R$ is a Ding-Chen ring, any absolutely pure complex $A$ is an exact complex with an upper bound on the flat dimensions of the $Z_nA$. It follows that $A$ itself is a complex of finite flat dimension.  So we are done.
\end{proof}

\begin{corollary}\label{cor-Ding-projective-complexes}
Let $R$ be a Ding-Chen ring. Then a complex $X$ is Ding projective if and only if each $X_n$ is a Ding projective $R$-module.
\end{corollary}

\begin{proof}
The ``only if'' part is easy to show. For the converse, Theorem~\ref{them-Ding-projective-modules} assures us that we only need to show $X$ equals the zero cycles of some exact complex of projective complexes. Certainly we can find an exact complex $\cdots \xrightarrow{d_3} P_2 \xrightarrow{d_2} P_1 \xrightarrow{d_1} P_0 \xrightarrow{\epsilon} X \xrightarrow{} 0$ with each $P_i$ a projective complex. So it is left to extend this complex to the right. First note that there is an obvious (degreewise split) short exact sequence $$0 \rightarrow X \xrightarrow{(1,d)} \bigoplus_{n \in \Z}D^{n+1}(X_n) \xrightarrow{-d+1} \Sigma X \rightarrow 0.$$
Now each $X_n$ is Ding projective. So we certainly can find a short exact sequence
$0 \rightarrow X_n \xrightarrow{\alpha_n} Q_n \xrightarrow{\beta_n} Y_n \rightarrow 0$ where $Q_n$ is projective and $Y_n$ is also Ding projective. This gives us another short exact sequence
\[
\bigoplus_{n \in \Z}D^{n+1}(X_n) \xrightarrow{ \bigoplus_{n \in \Z}D^{n+1}(\alpha_n)} \bigoplus_{n \in \Z}D^{n+1}(Q_n) \xrightarrow{\bigoplus_{n \in \Z}D^{n+1}(\beta_n)} \bigoplus_{n \in \Z}D^{n+1}(Y_n).
\]
Notice that $\bigoplus_{n \in \Z}D^{n+1}(Q_n)$ is a projective complex and we will denote it by $P^0$. Furthermore, let $\eta : X \rightarrow P^0$ be the composite
\[
X \xrightarrow{(1,d)} \bigoplus_{n \in \Z}D^{n+1}(X_n) \xrightarrow{\bigoplus_{n \in \Z}D^{n+1}(\alpha_n)} \bigoplus_{n \in \Z}D^{n+1}(Q_n).
\]
Then $\eta$ is an monomorphism since it is the composite of two monomorphisms. Moreover, setting $C^0 = \cok{\eta}$, it follows from the snake lemma that $C^0$ sits in the short exact sequence $0 \rightarrow   \Sigma X \xrightarrow{} C^0 \xrightarrow{}  \bigoplus_{n \in \Z}D^{n+1}(Y_n)  \rightarrow 0$. In particular, $C^0$ is an extension of $\bigoplus_{n \in \Z}D^{n+1}(Y_n)$ and $\Sigma X$, and so $C^0$ must be Ding projective in each degree since both of $\bigoplus_{n \in \Z}D^{n+1}(Y_n)$ and $ \Sigma X$ are such. Since $C^0$ has the same properties as $X$, we may continue inductively to obtain the desired resolution
$$0 \xrightarrow{} X \xrightarrow{\eta} P^0 \xrightarrow{d^0} P^1 \xrightarrow{d^1} P^2 \xrightarrow{d^2} \cdots$$
Finally, we paste the resolution together with $\cdots \xrightarrow{} P_2 \xrightarrow{d_2} P_1 \xrightarrow{d_1} P_0 \xrightarrow{\epsilon} X \rightarrow 0$
by setting $d_0 = \eta \epsilon$ and we are done.
\end{proof}

\section{Ding flat modules and complexes}\label{sec-flat}

As pointed out in~\cite{gillespie-ding}, there is also a natural notion of a Ding flat module, but it turns out to be equivalent to the notion of a Gorenstein flat module.  We recall the definition.

\begin{definition}\label{def-Ding flats}
We call an $R$-module $M$ \textbf{Ding flat} if there exists an exact complex of flat modules $$\cdots \rightarrow F_1 \rightarrow F_0 \rightarrow F^0 \rightarrow F^1 \rightarrow \cdots$$ with $M = \ker{(F^0 \rightarrow F^1)}$ and which remains exact after applying $A \otimes_R -$ for any absolutely pure (right) $R$-module $A$.

In the same way, we call a chain complex $X$ \emph{Ding flat} if there exists an exact complex of flat complexes $$\cdots \rightarrow F_1 \rightarrow F_0 \rightarrow F^0 \rightarrow F^1 \rightarrow \cdots$$ with $X = \ker{(F^0 \rightarrow F^1)}$ and which remains exact after applying $A \overline{\otimes} -$ for any absolutely pure complex $A$.
\end{definition}

\begin{proposition}
Let $R$ be any ring. A Ding flat module is nothing more than a Gorenstein flat module. Similarly, a chain complex $X$ is Ding flat if and only if it is Gorenstein flat.
\end{proposition}

A proof of the above proposition, for modules, goes back to~\cite[Lemma~2.8]{ding and mao 08}. We provide a new proof which is very quick and easy.

\begin{proof}
We give the proof for complexes but the same proof works for modules. Note that Ding flat complexes are clearly Gorenstein flat since injective complexes are absolutely pure. Conversely, suppose that $X$ is a Gorenstein flat complex. By definition, this means there exists an exact complex $$\mathbb{F} = \cdots \rightarrow F_1 \rightarrow F_0 \rightarrow F^0 \rightarrow F^1 \rightarrow \cdots$$ of flat complexes with $X = \ker{(F^0 \rightarrow F^1)}$  which remains exact after applying $I \,\overline{\otimes} -$ for any injective complex $I$.  We will show that this complex does in fact remain exact after applying $A \overline{\otimes} -$ for any absolutely pure (right) chain complex $A$. So let such an $A$ be given and note that it is equivalent to show that the map of complexes $A \overline{\otimes} Z_n\mathbb{F} \xrightarrow{} A \overline{\otimes} \mathbb{F}_n$ is a monomorphism for each $n$. Now let $A \hookrightarrow I$ be an embedding into an injective complex $I$ and note that this must be a pure monomorphism since $A$ is an absolutely pure complex. For each $n$ we have the commutative diagram:
$$\begin{CD}
A \overline{\otimes} Z_n\mathbb{F}  @>>> I \overline{\otimes} Z_n\mathbb{F} \\
@VVV @VVV       \\
A \overline{\otimes} \mathbb{F}_n @>>> I \overline{\otimes} \mathbb{F}_n \\
\end{CD}$$
The two horizontal arrows are monomorphisms since $A \hookrightarrow I$ is pure. The right vertical arrow is also a monomorphism since $I \overline{\otimes} \mathbb{F}$ is exact.  It follows that the left vertical arrow must also be a monomorphism.
\end{proof}

\begin{proposition}\label{them-Ding-flat-modules}
Let $R$ be a Ding-Chen ring. Then a module $M$ is Ding flat if and only if $M = Z_0F$ for some exact complex $F$ of flat $R$-modules. In the same way, a chain complex $X$ is Ding flat if and only if $X = Z_n \mathbb{F}$ for some exact complex $\mathbb{F}$ of flat complexes.
\end{proposition}

\begin{proof}
The proof is easy and very similar to the last few sentences of the proof of Theorem~\ref{them-Ding-projective-modules}. Briefly, say $\mathbb{F}$ is an exact complex of flat complexes; we wish to show that $A \overline{\otimes}\, \mathbb{F}$ is exact for any absolutely pure complex $A$. Certainly $F \overline{\otimes}\, \mathbb{F}$ is exact for any flat complex $F$, and, since $R$ is a Ding-Chen ring, any absolutely pure complex $A$ has finite flat dimension. We argue that $A \overline{\otimes} \mathbb{F}$ is exact for any such $A$.
\end{proof}

\begin{corollary}\label{cor-Ding-flat-complexes}
Let $R$ be a Ding-Chen ring. Then a complex $X$ is Ding flat if and only if each $X_n$ is a Ding flat $R$-module.
\end{corollary}

\begin{proof}
Using the above proposition, the proof of Corollary~\ref{cor-Ding-projective-complexes} carries over. Just replace the word \emph{projective} with the word \emph{flat} throughout the proof.
\end{proof}


\providecommand{\bysame}{\leavevmode\hbox to3em{\hrulefill}\thinspace}
\providecommand{\MR}{\relax\ifhmode\unskip\space\fi MR }
\providecommand{\MRhref}[2]{%
  \href{http://www.ams.org/mathscinet-getitem?mr=#1}{#2}
}
\providecommand{\href}[2]{#2}

\end{document}